\documentclass[12pt,a4paper]{article}

\usepackage{amscd, amsmath, amssymb, amsthm, amscd, latexsym}
\newtheorem{theorem}{Theorem}[section]
\theoremstyle{plain}

\newtheorem{proposition}[theorem]{Proposition}

\newtheorem{lemma}[theorem]{Lemma}

\newtheorem*{main}{Theorem}

\theoremstyle{definition}

\newtheorem{definition}[theorem]{Definition}
\newtheorem{remark}[theorem]{Remark}

\renewcommand{\a}{\alpha}
\newcommand{\ab}{\mathrm{ab}}
\newcommand{\Aut}{\mathrm{Aut}\,}
\renewcommand{\b}{\beta}

\newcommand{\bF}{\mathbb{F}}
\newcommand{\bFp}{\bF_p}
\newcommand{\bk}{\mathbf{k}}
\newcommand{\bS}{\mathbf{S}}

\newcommand{\bQ}{\mathbb{Q}}

\newcommand{\bZ}{\mathbb{Z}}

\newcommand{\Cf}{\textrm{cf.}\;}

\newcommand{\cO}{\mathcal{O}}

\newcommand{\ebar}{\bar{\eta}}

\newcommand{\fg}{\pi_1}

\newcommand{\fgX}{\fg(X)}

\newcommand{\fgXmg}{\fg(X,\m)^{\geo}}

\newcommand{\fgXm}{\fg(X,\m)}

\newcommand{\fgXmab}{\fg(X,\m)^{\ab}}

\newcommand{\Gal}{\operatorname{Gal}}
\newcommand{\geo}{\mathrm{geo}}

\newcommand{\GL}{\mathrm{GL}}

\newcommand{\Ker}{\operatorname{Ker}}
\newcommand{\kC}{k(C)}

\newcommand{\m}{\mathfrak{m}}

\newcommand{\PC}{P(C)}
\newcommand{\PtC}{P(\tC)}
\newcommand{\Hom}{\operatorname{Hom}}

\newcommand{\OF}{\cO_F}

\newcommand{\ol}[1]{\overline{#1}}

\newcommand{\opcit}{\textit{op.\,cit.}}

\newcommand{\ra}{\rightarrow}
\newcommand{\red}{\mathrm{red}}
\newcommand{\reflem}[1]{Lemma~{\rm \ref{#1}}}
\newcommand{\refprop}[1]{Proposition~{\rm \ref{#1}}}
\newcommand{\refthm}[1]{Theorem~{\rm \ref{#1}}}

\newcommand{\refsec}[1]{Section~{\rm \ref{#1}}}

\newcommand{\RII}{I\hspace{-.1em}\,I}

\newcommand{\SchZ}{\operatorname{Sch}(\bZ)}

\newcommand{\Spec}{\operatorname{Spec}}
\newcommand{\ssm}{\smallsetminus}
\newcommand{\tame}{\mathrm{tame}}
\newcommand{\tC}{\wt{C}}

\newcommand{\vphi}{\varphi}
\newcommand{\vphii}{\varphi^{-1}}

\newcommand{\wt}[1]{\widetilde{#1}}

\newcommand{\xra}{\xrightarrow}

\def\sn{\smallskip\noindent}

\title{Smallness of fundamental groups for arithmetic schemes}

\author{Shinya Harada and Toshiro Hiranouchi\footnote{
 The first author is 
 supported by the JSPS Fellowships for Young Scientists. 
 The second author is partially supported by GCOE, Kyoto University.
}}

\begin{document}

\pagenumbering{arabic}
\maketitle

\begin{abstract}
  The smallness is proved of
  \'etale fundamental groups for arithmetic schemes. 
  This is a higher dimensional analogue of
  the Hermite-Minkowski theorem.
  We also refer to the case of varieties
  over finite fields.
  As an application,
  we prove certain finiteness results
  of representations
  of the fundamental groups over
  algebraically closed fields.
\end{abstract}

%
%

\section{Introduction}
\label{Introduction}

The Hermite-Minkowski theorem 
is a remarkable result 
in algebraic number theory. 
It says that for a number field $F$  and 
a finite set $S$  of primes of $F$, 
there exist only finitely many 
extensions of $F$ 
unramified outside $S$ 
with given degree. 
By Galois theory 
we can interpret the theorem as 
the {\it smallness} of the Galois group $G_{F,S}$ 
of the maximal Galois extension of $F$ 
unramified outside $S$. 
In general, 
a profinite group is said to be {\it small}\/ 
if there exist only finitely many 
open subgroups of the group of given index. 
Our main results are the smallness of 
\'etale fundamental groups for some arithmetic schemes 
as follows.

\begin{main}[Th.\ \ref{thm:small}]
\label{arith-small}
  Let\/ $X$ be a connected scheme
  of finite type and dominant\footnote{In this paper we say that
 a morphism $f:X \to Y$ is dominant if
 $f(X')$ is dense in $Y$ for any irreducible
 component $X'$ of $X$.} 
  over the ring of integers $\bZ$.
  Then the \'{e}tale fundamental group
  $\fgX$ is small. 
  Equivalently, 
  there exist only finitely many 
  \'etale coverings of $X$ with given degree. 
\end{main}

On the contrary, 
non-complete varieties 
over a finite field, 
even function fields over a finite field 
in one variable with characteristic $p>0$ 
have so many finite extensions 
that the Hermite-Minkowski theorem no longer holds. 
In fact, 
the Artin-Schreier equations 
produce infinitely many extensions
of degree $p$ which ramify only at a place. 
Thus one can hope for the finiteness only for 
extensions of bounded degree 
with extra conditions, 
for example, 
fixing discriminants 
(see \cite{Goss}, Sect.~8.23). 
In \refsec{sec:modulus}, 
we use the fundamental groups with restricted ramification 
instead of discriminants. 
More precisely, 
we introduce a certain quotient  $\fgXm$  of 
the \'etale fundamental group  $\fgX$. 
It classifies \'etale coverings of  $X$  
which allow ramification along the {\it boundary} 
according to a given {\it modulus $\m$}. 
Here, the modulus  $\m$  is 
a collection of moduli 
associated with curves on  $X$. 
These coverings have the advantage that
they are stable under base change,
while coverings with restricted ramification
dealt in \cite{Hiranouchi} are not. 
We examine 
the smallness of the fundamental group with restricted ramification 
for varieties over a finite field as follows.


\begin{main}[Th.\ \ref{thm:G-small}]
 Let\/  $X$  be a connected variety over a finite field $k$.
  Assume that there exist an \'etale morphism 
  $X'\to X$ and a proper morphism 
  $X' \to Z$, where $Z$ is a curve over $k$.
  Then $\fgXm$  is small for any modulus  $\m$  on  $X$.  
\end{main}

As an application, 
in Section~\ref{sec:application}
we prove 
certain finiteness results of representations 
of the fundamental groups with restricted ramification over
an algebraically closed field. 

After writing up this paper, 
it was pointed out by Professor Y. Taguchi that 
one of our main theorems (Th.~\ref{thm:small})  
has been proved by G.~Faltings  
(\cite{Rational_Pts}, Chap.~VI, Sect.~$2$) 
assuming the scheme $X$ 
is affine and smooth over $\bZ$. 
However there is only a sketch of the proof, 
and our proof is more systematic 
and simpler than his 
by virtue of \reflem{extension-closed}. 
Thus we think that 
it is indispensable 
to leave Theorem~\ref{thm:small} 
 in this paper. 

Throughout this paper, 
a {\it number field}\/ 
is a finite extension field of the rational number field $\bQ$. 
We denote by  $\OF$  
the ring of integers of a number field $F$. 
For any field $K$,  we denote 
by $\ol{K}$ a separable closure of $K$
 and by $G_K$ the absolute Galois group $\Gal(\ol{K}/K)$
 of $K$. 

\medskip\noindent
{\it Acknowledgments.} 
  The authors are deeply indebted to
  Professor Yuichiro Taguchi
  for giving us the opportunity
  to try to the research problem in this paper
  and for giving many suggestions
  for the proof of \refthm{thm:small}.
  The first author thanks Professor Akio Tamagawa
  for indicating him 
  the reference
  Expos\'{e} \RII \ of \cite{SGA7-I}.
  He is also grateful to Professor John.\ S.\ Wilson
  for helpful comments on his question
  about \reflem{extension-closed}
  and for telling him the article \cite{Wilson}.
  He would also like to express
  his sincere gratitude to
  Professor Makoto Matsumoto
  for telling him another proof of
  \reflem{extension-closed} and
  for permitting him to include the proof
  in this paper. 
  The second author would like to thank Shuji Saito 
  for informing him of G.~Wiesend's recent works 
  \cite{Wiesend:CFT}, \cite{Wiesend:tame} and \cite{Wiesend:const}. 
  He also thanks Alexander Schmidt for his helpful comments 
  on the first version of this paper  
  and suggestion for improvement of Theorem~\ref{thm:small},
 and Moritz Kerz for answering his questions
 on the paper \cite{KS}.
 The authors thank the referee
 for many valuable suggestions 
 to improve \refsec{sec:modulus} 
 and the proof of \refthm{Finiteness of Reps ch=p} (ii).

  
\section{Smallness of fundamental groups}
\label{sec:Group Theory}
The aim of this section is to prove 
the smallness of the \'{e}tale fundamental group 
for a flat scheme over $\Spec(\bZ)$ (Th.~\ref{thm:small}). 
First, we 
interpret the Hermite-Minkowski theorem 
as the {\it smallness}\/ of a Galois group as follows.  

\begin{definition}[\Cf \cite{Field-arithmetic}, Sect.\ $16.10$,
    or of type (F) in Sect.\ 4.1 of \cite{Serre-Galcoh}]
  A profinite group $G$ is said to be {\it small}\/
  if there exist only finitely many open normal subgroups $H$
  with $(G:H) \le n$ for any positive integer $n$.
\end{definition}

For a number field $F$, 
the Galois group $G_{F,S}$ of the maximal Galois extension
of  $F$  unramified outside $S$  is small, 
where  $S$  is a finite set of primes  of  $F$.  
This fact is nothing other than
the Hermite-Minkowski theorem.

\begin{remark}
  The notion of a small profinite group
  is used in another meaning in \cite{Newelski}.
  These two notions are quite different.
  For example, 
  a direct product of countably many $\bZ/n \bZ$
  is small in the sense of \opcit,
  and it is not small in our meaning.
  In general,
  any countably-based profinite group
  is small in the sense of \opcit
\end{remark}

\begin{proposition}[\Cf \cite{Serre-Galcoh}, Chap.~III, Sect.~4.1, Prop.~8]
Let $G$  be a profinite group. 
The following conditions are equivalent:

\sn
$\mathrm{(i)}$ $G$  is small.

\sn
$\mathrm{(i')}$ There exist only finitely many
 open
 subgroups $H$
 with $(G:H) \le n$ for any positive integer $n$.

\sn
$\mathrm{(ii)}$ For every finite $G$-group $A$, 
$H^1(G,A)$  is a finite set. 

\sn
$\mathrm{(ii')}$ For every finite group $A$, 
 the set of continuous homomorphisms
 $\Hom(G,A)$ is a finite set. 
\end{proposition}

Secondly, we list basic properties of
small profinite groups.  

\begin{proposition}[\Cf \cite{Field-arithmetic}, Sect.\ 16.10]
\label{prop:small}
 Let $G$ be a profinite group.

\sn
$(\mathrm{i})$
 If $N$ is an open subgroup of $G$,
 then $G$ is small if and only if $N$ is small.

\sn
$(\mathrm{ii})$
 If $G$ is small and $N$ is a closed normal subgroup of $G$,
 then the quotient group $G/N$ is small.

\sn
$(\mathrm{iii})$
 If $G$ is topologically finitely generated,
 then $G$ is small.
\end{proposition}

\begin{lemma}
 \label{lem:small*small}
 If $G$ and $G'$ are small profinite groups,
 then their free product
 $G \ast G'$ is also small.
\end{lemma}
\begin{proof}
 For any finite group $A$,
 there is a bijection
$$
 \Hom( G \ast G' , A) \ra \Hom(G,A) \times \Hom(G',A)
$$
 which maps $f : G \ast G' \ra A$
 to the pair $(f|_G, f|_{G'})$, 
 where $f|_G$ and $f|_{G'}$ are 
 the restrictions of $f$ to $G$ and $G'$ respectively.
 Hence we have the assertion.  
\end{proof}

 The following lemma, \reflem{extension-closed}
 is a profinite group version of  
 Lemma 5 in \cite{Wilson}. 
We can prove it 
by the same method, 
namely by employing the theory of ``variety of groups''. 
Here we prove this
by purely profinite group theoretic argument.
This proof is due to Professor M.\ Matsumoto.
 First we prove a lemma,
 which is crucial for the proof of
 \reflem{extension-closed}.

\begin{lemma}
\label{claim-M}
 Let $1 \ra G' \ra G \xra{\vphi} G'' \ra 1$
 be an exact sequence of profinite groups,
 where $G'$ is a finite group and
 $G''$ is a small profinite group.
 For any positive integer $n$
 there exist
 only finitely many
 open normal subgroups $N$ of $G$
  with $(G : N) \le n$ and $G' \cap N = 1$.
\end{lemma}
\begin{proof}
 For any positive integer $n$, 
 let  $\bS_n$  be the set of 
 open normal subgroups $N$ of $G$  
 (we write it as $N\lhd G$)  
 with $(G : N) \le n$.  
 For  $N \in \bS_n$  with $N \cap G' = 1$, 
 the restriction $\vphi|_{N} : N \ra \vphi(N)$ of  $\vphi$ is an isomorphism
 and we have $(G'' : \vphi(N)) \le n$.
 Let $N''$ be an open normal subgroup of $G''$
 with $(G'' : N'') \le n$.
 By the smallness of $G''$,
 it is sufficient 
 to prove the finiteness of the set $\bS'_n$  
 of  $N \in \bS_n$ with  $N \cap G' = 1$ and $\vphi(N) = N''$.
 For any $N \in \bS'_n$,
 by composing the isomorphism $\vphii : N'' \ra N$
 and the inclusion $N \ra \vphii(N'')$,
 we have a section $N'' \ra \vphii(N'')$
 to $\vphii(N'') \xra{\vphi} N''$.
 This correspondence induces an injection $\bS'_n \ra \bS''_n$,
 where $\bS''_n$ is the set of sections $N'' \ra \vphii(N'')$
 to $\vphii(N'') \xra{\vphi} N''$.
 Now we prove that the set $\bS''_n$ is finite.
 Suppose that $\bS''_n \neq \emptyset$.
 Take a section $s_0 : N'' \ra \vphii(N'')$ in $\bS''_n$.
 Then we know that
 any section $s : N'' \ra \vphii(N'')$ in $\bS''_n$
 factors through the semi-direct product
 $G'\rtimes s_0(N'')$ which is a subgroup of $\vphii(N'')$.
 Let $K$ be the kernel of the conjugate action
 $N'' \ra s_0(N'') \ra \Aut(G')$,
 which is an open normal subgroup of $N''$.
 Then the restriction to $K$ of any section
 $s : N'' \ra G' \rtimes s_0(N'')$ in $\bS''_n$
 factors through $G' \times s_0(K)$.
 This is completely determined by $s_0$
 and the projection of $s|_{K}$ to $G'$.
 Since $K$ is small,
 there are only finitely many possibilities
 of the restriction of sections
 $s : N'' \ra G' \rtimes s_0(N'')$ to $K$.
 Since $K$ is open in $N''$,
 we conclude that $\bS''_n$ is a finite set.
\end{proof}

\begin{lemma}
\label{extension-closed}
 Let\/
 $ G' \ra G \ra G'' \ra 1$
 be an exact sequence of profinite groups.
 If $G'$ and $G''$ are small, then so is $G$.
\end{lemma}
\begin{proof}
 For any positive integer $n$,
 let  $\bS_n$  be the set
 as in Lemma\ \ref{claim-M}.
 For each  $N \in \bS_n$, we have
 $G' \cap N \lhd G'$ and $(G' : G' \cap N) \le n$.
 Note that $G' \cap N$  is also normal in  $G$.
 By taking quotient groups,
 we have an exact sequence
 $ G' / G' \cap N \ra G / G' \cap N \ra G'' \ra 1$.
 By the smallness of $G'$,
 there are only finitely many possibilities
 of open normal subgroups of $G'$
 with the form $G' \cap N$
 when $N$ runs through all the open normal subgroups of $G$
 with $(G : N) \le n$.
 Let $N'_1, \ldots, N'_r$ be
 such open normal subgroups of $G'$.
 Clearly we have
 \begin{align*}
   \bS_n
     &= \bigcup_{i=1}^r \left\{ N \lhd G \mid (G : N) \le n, \; N'_i \subset N \right\} \\
     &\simeq \bigcup_{i=1}^r \left\{ \bar{N} \lhd G/N'_i \mid (G/N'_i : \bar{N}) \le n \right\}. 
 \end{align*}
 Hence it is sufficient to replace $G$
 by $G/N'_i$ and assume $G'$ is a finite subgroup
 of $G$.
 Let $N_0$ be an open normal subgroup of $G$
 with $G' \cap N_0 = 1$
 (such $N_0$ exists
 since $G$ is Hausdorff and $G'$ is finite). 
 When $N$ runs through all the open normal subgroups of $G$
 with $(G : N) \le n$,
 there exist only finitely many possibilities of
 open normal subgroups of $G$ with the form $N \cap N_0$
 by Lemma\ \ref{claim-M}.
 Since $N \cap N_0$ is open in $G$,
 this implies that
 there are only finitely many possibilities of
 such $N$.
 Hence $G$ is small.
\end{proof}

We recall a homotopy exact sequence 
of \'etale fundamental groups. 

\begin{proposition}[{\cite{K-L}, Lem.\ 2}]
\label{homotopy exact sequence}
 Let $S$ be a connected normal and locally Noetherian scheme 
 with generic point  $\eta$. 
 Let $X$ be a scheme
 which is smooth and surjective over $S$
 and assume that its geometric generic fiber
 $X_{\ol{\eta}}$ is connected.
 Then the sequence of \'etale fundamental groups
$$
 \fg(X_{\ol{\eta}}) \to \fgX \to \fg(S) \to 1
$$
 is exact.
\end{proposition}

Finally, we prove the following theorem.
\begin{theorem}
\label{thm:small}
 Let  $X$  be a connected scheme
 of finite type and dominant
 over $\bZ$.
 Then the \'{e}tale fundamental group
 $\fgX$ is small.
\end{theorem}
\begin{proof}
  We shall first reduce to the case 
  in which $X$  is a normal scheme 
  smooth over $\OF$  of a number field $F$  
  and 
  the geometric generic fiber
 $X_{\ol{F}} := X \otimes_{\OF} \ol{F}$ is connected. 
  
  Since the scheme $X$ is Noetherian,  
  we have finite number of irreducible components $X_1,\ldots ,X_n$ of $X$. 
  The natural morphisms $X_i \to X$ induce 
  a morphism $\sqcup_i X_i \to X$ 
  which is an effective descent morphism.  
  By using the descent theory (\cite{SGA1}, Exp.~IX, Th.~5.1),
  the group $\fgX$ is a quotient of the group 
  generated by  
  $\fg(X_i)$ for all $i$ 
  and finitely many generators. 
  Thus, we may reduce to the case 
  $X$  is irreducible by
 \refprop{prop:small} and
 Lemma~\ref{lem:small*small}. 
  Since we have an isomorphism 
  $\fgX \simeq \fg(X_{\red})$ 
  for the reduced closed subscheme $X_{\red}$  of  $X$ 
  (\opcit, Exp.~IX, Prop.~1.7), 
  we may assume that $X$  is integral. 
  Let $X' \to X$ be the normalization morphism
  of $X$.
  Since $X$ is an excellent scheme,
  it is a finite morphism
 (\cite{EGAIV}, 7.8.3 (iii), 7.8.6 (i), (ii) or
 \cite{Liu}, 8.2.3).
  By the descent theory (\cite{SGA1}, Exp.~IX, Th.~5.1) 
  again, 
  we know that 
  the group $\fgX$  is a quotient 
  of the group generated by $\fg(X')$  
  and finitely many generators. 
  By \refprop{prop:small} and
 Lemma~\ref{lem:small*small},
  we may assume that
  $X$ is normal. 
  Let $F$ be 
  the algebraic closure of $\bQ$ in
  the function field of $X$. 
 Then the geometric generic fiber
  $X_{\ol{F}}$ is connected.
  Since $X$ is dominant over $\OF$, 
  it is flat over $\OF$
  (\cite{Liu}, Chap.\ 4, Prop.\ 3.9).
  For any open subscheme $U$ of $X$,
%
 the induced homomorphism $\fg(U) \ra \fgX$
 is surjective.
  Therefore, shrinking $X$ if necessary,
  we may assume that $X$ is smooth over $\Spec (\OF)$. 
  
  Let $S$  be the image of $X$ by the structure morphism 
  $X\to \Spec(\OF)$. 
  Since $X$  is flat over $\OF$,
  the set $S$ is open in $\Spec(\OF)$ (\cite{EGAIV}, 2.4.6). 
  Note that the complement of $S$ is a finite set 
  since $\OF$ is a Dedekind domain.
  By \refprop{homotopy exact sequence},
  we have the following exact sequence:
  $$
    \fg(X_{\ol{F}}) \ra \fgX \ra \fg(S) \ra 1.
  $$
  The fundamental group $\fg(S)$ is small
  by the Hermite-Minkowski theorem and 
  $\fg(X_{\ol{F}})$
  is topologically finitely generated 
  (\cite{SGA7-I}, Exp.\ \RII, Th.\ $2.3.1$). 
  Hence $\fgX$ is small by \reflem{extension-closed}.
\end{proof}

\begin{remark}
  The same argument works 
  for a connected scheme $X$  of finite type 
  over a local field $k$ with characteristic $0$. 
  Since the group $\fg(\Spec(k))$ 
  is known to be topologically finitely generated, 
  so is $\fg(X)$. 
\end{remark}

\section{Fundamental groups with modulus}
\label{sec:modulus}
The notion of coverings with restricted ramification 
defined in \cite{Hiranouchi}
is not stable under base change 
even if our attention restricts to tame covers 
(\Cf \cite{Schmidt}, Exam.\ 1.3). 
In this section, 
we shall introduce the notion of 
coverings of {\it ramification bounded by a modulus}\/ 
which is a slight modification of 
G.~Wiesend's tame coverings (\cite{Wiesend:tame}). 
Such coverings are stable under base change 
and form a Galois category.

Following \cite{KS}, 
we denote by $\SchZ$ 
the category of schemes separated and of finite type 
over $\Spec(\bZ)$. 
We call $X\in \SchZ$ {\it flat}\/  
if its structural morphism $X\to \Spec(\bZ)$ is flat, 
and {\it vertical} (or a {\it variety}) if the structural morphism 
factors through $\Spec(\bFp)$  for some prime number $p$. 
Note that an integral scheme $X\in \SchZ$ is either flat or vertical. 
A {\it curve}\/ is 
an integral scheme in $\SchZ$ of dimension $1$. 
A {\it curve on}\/ a scheme $X$  in $\SchZ$ is 
a closed subscheme of $X$ which is a curve. 
For any curve  $C$  on  $X$, 
let  
$\kC$  be the function field of  $C$ and  
$\tC$  the normalization of  $C$. 
If the curve $C$ is regular, 
there exists a unique regular curve $\PC\in \SchZ$ 
which is proper over $\bZ$  and contains $C$  
as a dense open subscheme (\Cf \cite{Kerz}, Prop.\ 1.6). 
We always identify closed points $x$ in $\PtC$  
with the normalized valuations in $k(C)$ associated with $x$. 
In particular, 
we write $k(C)_x$  the completion of $k(C)$ 
with respect to the valuation corresponding to a closed point $x \in \PtC$. 

\begin{definition}
\label{def:modulus}
  For every curve  $C$  on a scheme $X \in \SchZ$, 
  a {\it modulus $\m_C$ on  $C$} 
  is a finite set $(m_{C,x})_{x \in \PtC \ssm \tC}$ 
  of non-negative integers $m_{C,x}$ 
  associated with closed points $x$ in $\PtC \ssm \tC$. 
  A {\it modulus $\m = (\m_C)_{C \subset X}$ on}\/ $X$ 
  is a collection of moduli  $\m_C$  
  associated with curves $C$  on  $X$.  
\end{definition}

 Note that
 the sets of points and of curves
 on a scheme $X \in \SchZ$ 
 are at most countable,
 since a countable Noetherian ring has at most 
 countably many prime ideals  (cf. \cite{KS}, Sect.\ 7). 
 By using the notion of modulus, 
 we restrict the ramification of coverings as follows: 

\begin{definition}
\label{def:ram}
  ($\mathrm{i}$) 
  Let  $K$  be a 
  complete discrete valuation field, 
  $G_K$  the absolute Galois group of  $K$ 
  and  $L$  a separable extension field of $K$. 
  For any rational number  $m>-1$, 
  we say that the {\it ramification of  $L/K$  has bounded by  $m$} 
  if  $G_K^m \subset G_L$, 
  where $G_K^m$  is the $m$-th ramification subgroup of $G_K$ 
  in the upper numbering (\cite{Serre:68}, Chap.\ IV, Sect.\ 3). 

  \sn
  ($\mathrm{ii}$) 
  For each regular curve  $C$,  
  let $\m_C = (\m_{C,x})_x$ be a modulus on $C$. 
  A finite \'etale morphism  $C' \to C$  is said to be 
  of {\it ramification bounded by}\/  $\m_C$
  if the extension of complete discrete valuation fields  $k(C')_{x'} /\kC_x$  
  has ramification bounded by  $m_{C,x}$  
  for each  $x \in P(C) \ssm C$  and 
  for each  valuation  $x'$  of  $k(C')$  over  $x$. 

  \sn
  ($\mathrm{iii}$)
  Let $X\in \SchZ$ be connected, 
  and $\m$ a modulus on $X$. 
  A finite \'etale morphism  
  $Y \to X$  is said to be of 
  {\it ramification bounded by}  $\m$
  if for every curve  $C$  on  $X$  and 
  for each irreducible component  $C'$  of  $\tC\times_X Y$, 
  the ramification of the induced morphism  $C'\to \tC$  
  is bounded by $\m_C$.
\end{definition}


Let $X\in \SchZ$ be connected. 
Choose a closed point $x$ in $X$ and take a geometric point 
$\xi :\Spec(\Omega) \to x$, 
where $\Omega$  is a separably closed extension of the residue field at $x$. 
We define a fiber functor $F_x$  by $F_x(Y) = \mathrm{Hom}_X(\Spec(\Omega), Y)$
for any covering  $Y \to X$  of ramification bounded by $\m$.  

\begin{lemma}
The category of coverings over $X$  with ramification bounded by a modulus $\m$  
together with the fiber functor $F_x$ 
is a Galois category.
\end{lemma}
\begin{proof}
 We have to check the conditions (G1) - (G6) 
of \cite{SGA1}, Expos\'e V, Section 4.\ 
Let  $Y_1, Y_2 \to X$  be coverings 
of ramification bounded by $\m$, 
and  $C$ a curve on  $X$.  
The covering   
$\tC \times_X (Y_1 \times_X Y_2) \to \tC$  
 is of ramification bounded by  $\m_C$  (\cite{Hiranouchi}, Lem.\ 2.2). 
Thus, the category of these coverings over $X$ is 
closed under fiber products. 
Let  $Y_1\to Y_2$  and  $Y_2 \to X$  be two finite \'etale morphisms 
such that the composite  $Y_1 \to X$  
is a covering of ramification bounded by $\m$.  
Then the cover  $Y_2\to X$ is of ramification bounded by $\m$  
(\opcit, Lem.\ 2.2 and Lem.\ 2.4). 
 As in the proof of Theorem 2.4.2 in \cite{208},  
the above arguments 
imply 
(G1) and (G2), namely, 
the existence of fiber products and quotients respectively 
in the category. 
The assertions (G3) - (G6)  
can be deduced by the straight forward manner.
\end{proof}

 We denote by  $\fgXm$ 
 the fundamental group associated to the Galois category 
of coverings of $X$ with ramification bounded by $\m$. 
The notion of the coverings above is stable under base change 
as follows: 
Let $f:X'\to X$ be a morphism of connected schemes in $\SchZ$ and  
$\m = (\m_C)_{C \subset X}$ a modulus on $X$. 
For each curve $C'$ on $X'$ and 
closed point $x'$ in $P(\wt{C'}) \ssm \wt{C'}$, 
if $f(C')$ is a closed point of $X$ then put $\m_{C'} := 0$.
On the other hand, 
if the topological closure $\ol{f(C')} = C$ of $f(C')$  is a curve on $X$, then 
put $m_{C',x'} := e_{x'/x}m_{C,x}$ for $f(x') = x$ , 
where $e_{x'/x}$  is the ramification index 
of the extension $k(C')_{x'}/k(C)_x$. 
We put $f^{\ast}\m := (\m_{C'})_{C'\subset X'}$, 
where $\m_{C'} = (m_{C',x'})$.   

\begin{lemma}
\label{lem:modulus} 
Let $f:X'\to X$ be a morphism of connected schemes in $\SchZ$ 
and $\m$ a modulus on $X$. 
 Then the morphism $f$ induces a group homomorphism 
$\fg(X',f^{\ast}\m) \to \fgXm$.
\end{lemma}
\begin{proof}
Let  $Y\to X$  be a covering of $X$ 
of ramification bounded by $\m$.  
It is enough to show that 
the induced covering $Y\times_X X' \to X'$ of $X'$ 
has the modulus $f^{\ast}\m = (\m_{C'})_{C' \subset X'}$ 
for each curve $C'$ on $X'$. 
If  $f(C')$  is a closed point of  $X$,  
the covering $Y\times C' \to C'$ 
induces a separable base field extension and unramified.  
If the topological closure $\ol{f(C')}$ of $f(C')$  
is a curve  $C$  on  $X$, 
it is a base change of  $\tC\times_X Y \to \tC$ 
whose ramification is bounded by $\m_C$. 
Thus, the induced covering 
$\wt{C'} \times_X Y \to \wt{C'}$ has ramification bounded 
by $\m_{C'}$.
\end{proof}

For two moduli $\m = (\m_C)_{C}$ and $\m' = (\m'_C)_{C}$ on $X$, 
we write $\m'\ge \m$  
if we have $m'_{C,x} \ge m_{C,x}$ 
for each curve $C$ on $X$ and $x\in \PtC \ssm \tC$. 
Thereby we have a surjection $\fg(X,\m') \to \fg(X,\m)$. 
Using this notation, 
we have the following homotopy exact sequence. 

\begin{lemma}\label{lem:hes}
  Let $S\in \SchZ$ be connected normal with generic point $\eta$, 
  and $f:X\to S$  a smooth morphism in $\SchZ$. 
  Assume that the geometric generic fiber $X_{\ebar}$ of $f$ 
  is connected and the map $f$ has a section  $s: S \to X$. 
  Then,
  for any modulus $\m$ on $X$
  there is a modulus $\m'\ge \m$ on $X$ such that 
  the sequence of fundamental groups
  $$
    \fg(X_{\ebar}) \to \fg(X,\m') \to \fg(S, s^{\ast}\m) \to 1
  $$
  is exact.
\end{lemma}
\begin{proof}
  We basically follow the proof of Lemma 2 in \cite{K-L}. 
  For each curve  $D$  on  $S$, 
  its image  $s(D)$  of the section  
  $s: S\to X$  is a curve on  $X$.  
  Therefore we have 
  $s^{\ast}\m = (\m_{s(D)})_{D \subset S}$.   
  By Lemma~\ref{lem:modulus}, 
  the section $s:S\to X$ 
  induces a homomorphism  
  $\fg(S,s^{\ast}\m) \to \fgXm$. 
  In general, we do not have a homomorphism 
  $\fg(X,\m) \to \fg(S,s^{\ast}\m)$. 
  By the very definition of the modulus  
  $f^{\ast}s^{\ast}\m$, we can take a modulus $\m'$  on $X$  
  with $\m'\ge \m, \m'\ge f^{\ast}s^{\ast}\m$ 
  and satisfying $\m'_{s(D)} = \m_{s(D)}$  
  for each curve $D$ on $S$.  
 By \reflem{lem:modulus}
  we obtain a composition of homomorphisms
 $\beta: \fg(X,\m') \to \fg(X,f^{\ast}s^{\ast}\m)
 \to \fg(S,s^{\ast}\m)$. 
  By \refprop{homotopy exact sequence}, 
  we have the following commutative diagram
  $$
  \begin{CD}
     \fg(X_{\ebar}) @>>>  \fg(X) @>>>  \fg(S) @>>>  1\ \\
      @| @VVV @VVV \\
     \fg(X_{\ebar}) @>\a>>  \fg(X,\m') @>\b>>  \fg(S,s^{\ast}\m) @>>>  1,
  \end{CD}
  $$
  whose top row is exact.  
  Therefore  $\beta$  is surjective, and  $\b\circ \a = 0$. 
  To prove the exactness of the bottom row, we must show that 
  for any connected covering  $X'\to X$  
  with modulus  $\m'$  
  which admits a section over  $X_{\ebar}$,  
  there exists a connected covering  $S' \to S$  
  with modulus  $s^{\ast}\m$  
  such that  $X' = X\times_S S'$. 
  By the exactness of the top row of the above diagram, 
  we obtain a finite \'etale cover  $S'\to S$  
  with $X' = X\times_S S'$. 
  For each curve $D$  on $S$, 
  the covering $X'\times_X s(D) \to s(D)$ 
  is ramification bounded by $\m'_{s(D)} = \m_{s(D)}$. 
 Thus, 
 the covering $S' \to S$ has ramification
 bounded by the modulus $s^{\ast}\m$
 and the assertion follows.
\end{proof}

Now, we examine 
the smallness of fundamental groups with modulus. 
The fundamental group  $\fgXm$  
is a quotient of $\fgX$.  
Thereby it is small for any flat and connected scheme $X\in \SchZ$  
by \refprop{prop:small}, (ii)  and Theorem \ref{thm:small}. 
When the scheme  $X$  is vertical, 
we need the assumption on  $X$  
as in \cite{KS}, Theorem\ 8.2. 
%
\begin{theorem}
  \label{thm:G-small}
  Let\/  $X$  be a connected variety over a finite field $k$. 
  Assume that there exist an \'etale morphism 
  $X'\to X$ and a proper morphism 
  $X' \to Z$, where $Z$ is a curve over $k$ 
  with generic point $\eta$. 
  Then $\fgXm$ and  
  the geometric part $\fgXmg := \Ker(\fgXm \to G_k)$ 
  of $\fgXm$ are small
  for any modulus  $\m$  on  $X$.
\end{theorem}

\begin{proof}
  For any \'etale morphism $f:X'\to X$, 
  the induced map $\fg(X')\to \fg(X)$  has finite cokernel 
  and 
  so is $\fg(X',f^{\ast}\m)\to \fg(X,\m)$. 
  Hence we may assume that 
  we have a proper morphism $g:X\to Z$ over $k$. 
 We may further assume that $g:X\to Z$ is surjective.
  By the descent theory for \'etale fundamental groups, 
  we may assume that the variety  $X$  is normal and 
  the geometric generic fiber $X_{\ebar}$  of  $g:X\to Z$  is connected 
  as in the proof of Theorem\ \ref{thm:small}. 
 Let $U$ be the smooth locus of $g : X \to Z$
 and $i : U \ra X$ the open immersion.
 Note that $U$ is an open subscheme of $X$
 (\cite{EGAIV}, 17.15.12),
 and the induced morphism $g|_U:U \to Z$
 is still surjective.
 By \opcit, 17.16.3, the restriction $g|_U:U\to Z$ admits a 
  section over an \'etale open $Z'$ of $Z$. 
 Hence we may assume that
 the smooth surjective morphism $g|_U:U\to Z$
 has a section $s : Z \ra U$.
 Now we replace the modulus $\m$
 by $\m''$ so that it satisfies
 $\m'' \ge \m,$ $\m'' \ge g^{\ast}s^{\ast}i^{\ast}\m$
 and $\m''_{\ol{s(Z)}} = (g^{\ast}s^{\ast}i^{\ast}\m)_{\ol{s(Z)}}=\m_{\ol{s(Z)}}$,
 where $\ol{s(Z)}$ is the topological closure of $s(Z)$ in $X$.
 Since the geometric generic fiber  $U_{\ebar}$  is irreducible, 
 we have the following exact sequence by \reflem{lem:hes}
 for some modulus $\m'\ge i^{\ast}\m$  on $U$: 
  $$
   \fg(U_{\ebar}) \to \fg(U,\m')  \to \fg(Z,s^{\ast}i^{\ast}\m)\to 1. 
  $$
 The natural homomorphisms
 $\fg(U,\m') \ra \fg(X,\m)$ and
 $\fg(U_{\ebar}) \ra \fg(X_{\ebar})$
 are surjective.
 Hence the sequences
  $$
   \fg(X_{\ebar}) \to \fg(X,\m)  \to \fg(Z,s^{\ast}i^{\ast}\m)\to 1
  $$
 and
  $$
    \fg(X_{\ebar}) \to \fg(X,\m)^{\geo} \to \fg(Z,s^{\ast}i^{\ast}\m)^{\geo} \to 1
  $$
 are also exact.
  The fundamental group
 $\fg(Z,s^{\ast}i^{\ast}\m)^{\geo}$  is small 
  (\Cf \cite{Goss}, Th.~8.23.5) and 
  $\fg(X_{\ebar})$  is topologically finitely generated 
  (\cite{SGA1}, Exp.~X, Th.~2.9) 
  since the fiber  $X_{\ebar}$  is proper.
  Thus, the assertion follows from \reflem{extension-closed}.
\end{proof}
\begin{remark}
  It is known that the tame fundamental groups for curves 
  over a finite field are topologically finitely generated. 
  By the same manner as in the proof of
  the above theorem, 
  for a variety $X$ over a finite field 
  which is \'etale locally proper over a curve,  
  we can prove that the tame fundamental group $\fg^{\tame}(X)$  
  is also topologically finitely generated.
\end{remark}

\section{Application to representations of fundamental groups}
\label{sec:application}
 As an application of the Hermite-Minkowski theorem
 and the finiteness of ray class groups, 
 the following finiteness of Galois representations 
 over an algebraically closed field $\bk$ is obtained
 (for the geometric version of this result,
 see \cite{M-T}, Th.\ $4$ $(\mathrm{i})$). 
 Here, we always consider the general linear group $\GL_d(\bk)$ 
 as a topological group with the discrete topology.

\begin{theorem}[\cite{A-G-D-R-G}, see also \cite{M-T}, Th.\ $1$]
\label{conj.1dim} 
 Let $\bk$  be an algebraically closed field and
 $F$ a number field.
 Suppose that the characteristic of $\bk$ is $0$.
 For any positive integer $d$, 
 there exist only finitely many isomorphism classes
 of continuous semisimple representations
 $\rho:G_F \to \GL_d(\bk)$ with bounded Artin conductor.
\end{theorem}
In the case where the characteristic of
$\bk$  is $p> 0$,
 H.~Moon and Y.~Taguchi proved
 the finiteness of mod $p$ Galois representations 
 with solvable images over  $\bk$
 both for the number field case
 and for the function field case 
(\cite{M-T}, Th.\ $2$, Th.\ $4$ $(\mathrm{ii})$). 
 For the function field case,
 the finiteness has been obtained in almost all the cases
 by G.~B{\"o}ckle and C.~Khare
 (see \cite{B-K}).  

 By using the fundamental groups
 with modulus 
 defined in Section~\ref{sec:modulus} 
 instead of the Artin conductor,  
 the finiteness of mod $p$ Galois representations
 is equivalent
 to the finiteness of representations
 of the fundamental group 
 $\fg(X,\m)$,
 where
 $X$ is a curve in $\SchZ$
 and $\m$ is a modulus on $X$. 
 
\begin{definition}
\label{def:commutator}
For any topological group $G$ and non-negative integer $r$, 
define the {\it $r$-th commutator subgroup} $G^{(r)}$ of $G$ 
by the topological closure of 
the commutator subgroup $[G^{(r-1)},G^{(r-1)}]$  
in $G^{(r-1)}$ for $r\ge 1$, and $G^{(0)}:=G$. 
\end{definition}

 Now we prove the finiteness of $\fgXm/\fgXm^{(r)}$
 by reducing to the one-dimensional case.
 The following proposition is essential for the proof.

\begin{proposition}[\cite{KS}, Th.\ 2.9 (i)] 
\label{prop:KS2.9}
Let $X\in \SchZ$ be connected normal. 

\sn 
$\mathrm{(i)}$ If $X$ is flat, 
 then there exists a horizontal curve $C$ on $X$
 such that the induced homomorphism
$\fg(C)^{\ab} \ra \fg(X)^{\ab}$ has an open image.
 
\sn
$\mathrm{(ii)}$  
 If $X$ is vertical, 
 assume that there exist an \'etale open $X' \ra X$ and
 a proper generically smooth morphism $X' \ra Z$
 to a regular curve $Z$.
 Then we find a curve $C$ on $X$
 with the same property as in $\mathrm{(i)}$.
\end{proposition}

\begin{lemma}
\label{lem:solv}
Let  $X \in \SchZ$  be connected normal 
and $\m$ a modulus on $X$. 

\sn
$(\mathrm{i})$ 
 If $X$  is flat, then 
 $\fgXm/ \fgXm^{(r)}$  is finite 
 for any integer $r \ge 1$.

\sn
$(\mathrm{ii})$ 
 If $X$  is vertical over a finite field $k$ 
 we assume that there exist an \'etale open $X' \ra X$ and
 a proper generically smooth morphism $X' \ra Z$
 to a regular curve $Z$ over $k$. 
 Then 
 $\fgXm^{\geo}/\fgXm^{\geo,(r)}$ is finite 
 for any integer $r \ge 1$.  
\end{lemma}
\begin{proof}
\sn
$(\mathrm{i})$ 
In the case of  $r=1$, 
we write $\fgXm/\fgXm^{(r)} = \fgXmab$. 
By Proposition\ \ref{prop:KS2.9}
there exists a flat curve $C$ on $X$
such that the induced homomorphism
$\fg(C,\m_C)^{\ab} \ra \fgXm^{\ab}$ 
has open image.
Hence, 
we have the following commutative diagram
$$
\begin{CD}
 \fg(C)^{\ab} @>>> \fg(X)^{\ab}\\
@VVV @VVV \\
\fg(C,\m_C)^{\ab} @>>> \fgXm^{\ab}.
\end{CD}
$$
 Since $\fg(C,\m_C)^{\ab}$ is finite 
 and the homomorphism $\fg(C, \m_C)^{\ab} \ra \fgXmab$
 has finite cokernel,
 $\fgXmab$ is also a finite group.

By induction on $r$, 
we may assume that $\fgXm/ \fgXm^{(r-1)}$  
is finite. 
Take a Galois cover  $f : X'\to X$  corresponding to
  $\fgXm^{(1)}$. 
 We have a surjection  $\fg(X',\m') \to \fgXm^{(1)}$
  for   the modulus  $\m':= f^{\ast}\m$  on  $X'$, 
and it induces a surjection  
$$
\fg(X',\m')/\fg(X',\m')^{(r-1)} \to \fgXm^{(1)}/\fgXm^{(r)}.
$$
Thus  the assertion follows from 
the hypothesis and the finiteness of $\fgXmab = \fgXm/\fgXm^{(1)}$.

\sn
$(\mathrm{ii})$ 
In the case of  $r=1$, 
we have the desired result
by the same manner as in (i). 
For $r>1$, 
by replacing $k$ with a finite extension field over $k$
we may assume that the exact sequence
$$
 1 \ra \fgXm^{\geo} \ra \fgXm \ra G_k \ra 1
$$
splits.
 Since $\fgXm^{\geo,(r)}$ is a characteristic
 subgroup of $\fgXm^{\geo}$,
 the semidirect product $\fgXm^{\geo,(r)} \rtimes G_k$
 exists and it is a normal subgroup of $\fgXm$.
 Let $f:X' \ra X$ be a finite Galois cover
 corresponding to the normal subgroup
 $\fgXm^{\geo,(1)} \rtimes G_k$.
 As in (i) this induces a surjective homomorphism
$$
\fg(X',\m')^{\geo}/\fg(X',\m')^{\geo,(r-1)} \to \fgXm^{\geo,(1)}/\fgXm^{\geo,(r)}.
$$
 Hence we obtain the finiteness of
 $\fgXm^{\geo,(1)}/\fgXm^{\geo,(r)}$
 by induction on $r$.
\end{proof}

\begin{theorem}
\label{Finiteness of Reps ch=p} 

Let\/ $\bk$  be an algebraically closed field 
with arbitrary characteristic,
%
 $X \in \SchZ$ connected normal, 
 $d$ a positive integer and 
$\m$ a modulus on  $X$.

\sn
$(\mathrm{i})$ 
 If $X$ is flat, then
 there exist only finitely many isomorphism classes
 of semisimple continuous representations
 $\rho: \fgXm \ra \GL_d(\bk)$ 
 with solvable images.

\sn
$(\mathrm{ii})$
 If  $X$  is vertical over a finite field $k$ 
 we assume that
 there exist an \'etale open $X' \ra X$ and
 a proper generically smooth morphism $X' \ra Z$
 to a regular curve $Z$ over $k$.
 Then
 there exist only finitely many isomorphism classes
 of semisimple continuous geometric representations
 $\rho: \fgXm \ra \GL_d(\bk)$ with solvable images.
\end{theorem}
\begin{proof}
(i) 
 For any group $G$,
 {\it the solvability class of} $G$ is the minimal integer
 $i \ge 0$ such that $G^{(i)} = 1$.
 Note that
 every solvable subgroup $G$ of $\GL_d(\bk)$ has
 solvability class $\le s$,
 where $s$ is a positive integer depending only on $d$ 
 (\cite{Suprunenko}, Chap.\ V, Sect.\ $20$, Th.\ $8$).
 Thus every continuous representation
 $\rho : \fgXm \ra \GL_d(\bk)$ with solvable image 
 factors through
 the quotient group of $\fgXm$ by
 $\fgXm^{(s)}$.
 Note that
 this quotient group is finite
 by Lemma~\ref{lem:solv} (i).
 Thus the assertion $(\mathrm{i})$ follows.

\sn
$(\mathrm{ii})$
 By replacing the finite field $k$
 with a finite extension field, 
 we may assume that $X$ has a $k$-rational point. 
 Thus,
 the canonical exact sequence
$$
 1 \ra \fgXm^{\geo} \ra \fgXm \ra G_k \ra 1
$$
 splits.
 Then we see that
 the restriction map
 from the set of isomorphism classes of
 continuous semisimple geometric representations
 $\rho : \fgXm \ra \GL_d(\bk)$
 into the set of isomorphism classes of
 continuous semisimple representations of $\fgXm^{\geo}$
 into $\GL_d(\bk)$
 is bijective.
 As we see in (i),
 every continuous semisimple representation
 $\rho : \fgXm^{\geo} \ra \GL_d(\bk)$
 factors through the quotient group
 $\fgXm^{\geo}/\fgXm^{\geo,(s)}$ for some $s$,
 which is finite by \reflem{lem:solv} (ii).
\end{proof}

\begin{theorem}
\label{Arith : Finiteness of Reps} 
Let\/ $\bk$  be an algebraically closed field 
with characteristic $0$, 
$X \in \SchZ$ connected normal, 
 $d$ a positive integer and 
 $\m$ a modulus on  $X$.

\sn
$(\mathrm{i})$
  If  $X$  is flat, then
  there exist only finitely many isomorphism classes
  of semisimple continuous representations
  $\rho: \fgXm \ra \GL_d(\bk)$. 

\sn
$(\mathrm{ii})$
 If $X$  is vertical over a finite field $k$
 we assume that
 there exist an \'etale open $X' \ra X$ and
 a proper generically smooth morphism $X' \ra Z$
 to a regular curve $Z$ over $k$.
 Then there exist only finitely many isomorphism classes
 of semisimple continuous geometric representations
 $\rho: \fgXm \ra \GL_d(\bk)$.
\end{theorem}
\begin{proof}
 First note that, in the case of
 \refthm{Arith : Finiteness of Reps} (ii),
 by replacing the finite field $k$
 with a finite extension,
 the set of isomorphism classes of
 continuous semisimple geometric representations
 of $\fgXm$ into $\GL_d(\bk)$
 is equal to
 the set of isomorphism classes of
 continuous semisimple representations
 of $\fgXmg$ into $\GL_d(\bk)$.
 It is sufficient
 for the proof of $(\mathrm{i})$
 (resp.\ $(\mathrm{ii})$)
 to show that
 there are only finitely many possibilities
 of open normal subgroups of $\fgXm$
 (resp. $\fgXmg$)
 which appear as the kernels of
 semisimple continuous
 representations
 $\rho : \fgXm \ra \GL_d(\bk)$
 (resp.\ $\rho : \fgXmg \ra \GL_d(\bk)$)
 since there exist only finitely many
 isomorphism classes of
 irreducible representations of a finite group over $\bk$
 (\Cf \cite{MethodsI}, $(21.25)$).
 Let $\rho : \fgXm \ra \GL_d(\bk)$
 (resp. $\rho : \fgXmg \ra \GL_d(\bk)$)
 be
 a continuous semisimple representation.
 By Jordan's theorem
 (\Cf \cite{Suprunenko}, Chap.\ VI, Sect.\ $24$, Th.\ $3$)
 there exists an open normal subgroup $N$
 of $\fgXm$ (resp. $\fgXmg$)
 containing the kernel $\Ker(\rho)$ of $\rho$ 
 such that $N/\Ker(\rho)$ is abelian
 and that $(\fgXm:N) \le J(d)$
 (resp. $(\fgXmg:N) \le J(d)$),
 where $J(d)$ is a positive integer depending only on $d$.
 Since $\fgXm$ (resp. $\fgXmg$) is small,
 there exist only finitely many such $N$.
 Let $f : X' \ra X$ be the Galois covering
 corresponding to $N$.
 Then 
 there exists
 a surjective homomorphism $\fg(X',\m') \ra N$
 (resp. $\fg(X',\m')^{\geo} \ra N$)  
 for the modulus  $\m':= f^{\ast}\m$  on  $X'$  
 and it induces a surjection
 $\fg(X',\m')^{\ab} \ra N^{\ab}$
 (resp. $\fg(X',\m')^{\ab,\geo} \ra N^{\ab}$).
 Thus $N^{\ab}$ is finite by Lemma \ref{lem:solv}.
 Since $N/\Ker(\rho)$ is abelian,
 the commutator subgroup of $N$
 is contained in $\Ker(\rho)$.
 This proves the assertion.
\end{proof}


\providecommand{\bysame}{\leavevmode\hbox to3em{\hrulefill}\thinspace}
\providecommand{\href}[2]{#2}


\vspace{0.5cm}

\noindent
 Shinya Harada \\
 Graduate School of Mathematics \\
 Kyushu University \\
 6-10-1, Hakozaki, Higashiku, Fukuoka-city, 812-8581
 Japan \\
 JSPS Research Fellow, \\
{\tt s.harada@math.kyushu-u.ac.jp}

\vspace{0.5cm}

\noindent
 Toshiro Hiranouchi \\
 Research Institute for Mathematical Sciences,\\ 
 Kyoto University, \\
 Kyoto 606-8502 Japan,\\ 
{\tt hira@kurims.kyoto-u.ac.jp }

\end{document}